\documentclass[12pt,a4paper]{article}
\usepackage{graphicx}
\usepackage{amsmath}
\usepackage[cp1251]{inputenc}
\usepackage[T2A]{fontenc}
\usepackage[russian]{babel}
\usepackage{a4}
\usepackage{amsfonts}
\usepackage{amssymb}

\newtheorem{theorem}{Theorem}

\newtheorem{lemma}[theorem]{Lemma}

\newenvironment{proof}[1][Proof]{\textbf{#1.} }{\ \rule{0.5em}{0.5em}}
\numberwithin{equation}{section}

\begin{document}


\textbf{On Fourier multipliers in function spaces with partial
H\"{o}lder condition and their application to the linearized
Cahn-Hilliard equation with dynamic boundary conditions.}

\bigskip

S.P.Degtyarev

\bigskip

\bigskip

\textbf{Institute of Applied Mathematics and Mechanics, National
Academy of Sciences of Ukraine, Donetsk}

\bigskip

We give relatively simple sufficient conditions on a Fourier multiplier, so
that it maps functions with the H\"{o}lder property with respect to a
part of the variables to functions with the H\"{o}lder property with respect
to all variables. With the using of these sufficient conditions we prove the
solvability in H\"{o}lder classes of the initial-boundary value
problems for the linearized Cahn-Hilliard equation with dynamic boundary
conditions of two types. For the solutions of these problems Schauder
estimates are obtained.

The final expanded version of this paper is available at AIMS Journals, 
Evolution Equations and Control Theory (EECT) at 

http://www.aimsciences.org/journals/displayArticlesnew.jsp?paperID=11870

\bigskip

MSC \ \ \ 42B15, 42B37, 35G15, 35G16, 35R35

\bigskip

\bigskip

\section{Introduction.}

\label{s1}

The starting point for this paper is the paper by O.A.Ladyzhenskaya
\cite{1}. (see also \cite{2}). The original idea of the reasonings
from \cite{1} and \cite{2}, as it was pointed in these papers, is
taken from \cite{Herm}, Theorem 7.9.6 and this idea is based on the
Littlewood-Paley  decomposition. However, papers \cite{1} and \cite{2}
deal with more general than in \cite{Herm} case of anisotropic
H\"{o}lder spaces. Moreover, paper \cite{1} gives some simple
sufficient conditions on a Fourier multiplier to provide bounded
mapping with this multiplier in anisotropic H\"{o}lder spaces.
These conditions can be easily verified in particular problems for
partial differential equations as it was demonstrated in \cite{1}.


It is important that the sufficient conditions from \cite{1} can be
comparatively easily verified in the case when a multiplier is an
anisotropic-homogeneous  function of degree zero (or it is close to
such a function in some sense). And it is also fundamentally
important in \cite{1} that the anisotropy of a H\"{o}lder space
where a multiplier acts must coincide with the anisotropy of the
multiplier (see Theorem \ref{T1.1} below).


Though the results of \cite{1} are applicable to a brad class of
problems for partial differential equations (as it was pointed out
in \cite{1}), they are still not applicable to many problems where
the anisotropy of a H\"{o}lder space does not coincide with the
anisotropy of a multiplier or where one should consider a
H\"{o}lder space of functions with the H\"{o}lder
conditions with respect only to a part of independent variables.
Among such problems we first mention "nonclassical" statements
connected to so called "Newton polygon" - see, for example,
\cite{3}-\cite{R1}. A particular subclass of such problems is the
class of problems for parabolic and elliptic equations with highest
derivatives in boundary conditions including the Wentsel conditions.
This subclass includes also problems with dynamic boundary
conditions- \cite{3}, \cite{4}, \cite{P2}-\cite{8.1}. As it is well
known, such problems for parabolic and elliptic equations are not
included in the standard general theory of parabolic boundary value
problems - see, for example, \cite{6}-\cite{6.2}.

In this paper we apply results about Fourier multipliers in spaces
of functions with partial H\"{o}lder condition to
initial-boundary value problems for linearized Cahn-Hilliard
equation with dynamic boundary conditions of two types. To our
knowledge earlier such problems were under investigation only in
spaces of functions with integrable derivatives - see, for example,
\cite{3}, \cite{4}, \cite{P2}-\cite{8.1}. Not that at the
considering of such problems we have the situation when the
anisotropy of a multiplier does not coincide with the anisotropy of
the corresponding H\"{o}lder space. In this situation we
consider smoothness of functions with respect to each of their
variables separately.

Note also that problems with dynamic boundary conditions arise as a
linearization of many well-known free boundary problems such as the
Stefan problem, the two-phase filtration problem for two
compressible fluids (the parabolic version of the Muskat-Verigin
problem), the Hele-Shaw problem, the classical evolutionary
Muskat-Verigin problem for elliptic equations.


Besides, the studying of smoothness of solutions of some problems
with respect to only a part of independent variables (including
obtaining corresponding Schauder's estimates) has it's own history
and it is an important direction of investigations. In particular,
we deal with such situation when considering semigroups of operators
with parameter $t>0$ and with a generator defined on some
H\"{o}lder space see, for example, \cite{35}- \cite{52}.

In all such cases we can use a theorem about multipliers in spaces
of functions with smoothness with respect to only a part of
independent variables because it permits to consider the smoothness
with respect to each variable separately.

Let us introduce now some notation and formulate an assertion, which
is a simple consequence of Theorem 2.1 and lemmas 2.1, 2.2 from
\cite{1} because we need it for references.

Let for a natural number  $N$

\begin{equation}
\gamma \in (0,1),\alpha =(\alpha _{1},\alpha _{2},...,\alpha _{N}),\quad
\alpha _{1}=1,\alpha _{k}\in (0,1],k=\overline{2,N}.  \label{1.1}
\end{equation}

Denote by $C^{\gamma \alpha }(R^{N})$ the space of continuous in
$R^{N}$ functions $u(x)$ with the finite norm

\begin{equation}
\left\Vert u\right\Vert _{C^{\gamma \alpha }(R^{N})}\equiv \left\vert
u\right\vert _{R^{N}}^{(\gamma \alpha )}=\left\vert u\right\vert
_{R^{N}}^{(0)}+\sum_{i=1}^{N}\left\langle u\right\rangle
_{x_{i},R^{N}}^{(\gamma \alpha _{i})},  \label{1.2}
\end{equation}
where

\begin{equation}
\left\vert u\right\vert _{R^{N}}^{(0)}=\sup_{x\in R^{N}}\left\vert
u(x)\right\vert ,\quad \left\langle u\right\rangle _{x_{i},R^{N}}^{(\gamma
\alpha _{i})}=\sup_{x\in R^{N},h>0}\frac{\left\vert
u(x_{1},...,x_{i}+h,...,x_{N})-u(x)\right\vert }{h^{\gamma \alpha _{i}}}.
\label{1.3}
\end{equation}


Along with the spaces $C^{\gamma \alpha }(R^{N})$ with the exponents
$\gamma \alpha _{i}<1$ we consider also spaces
$C^{\overline{l}}(R^{N})$, where
$\overline{l}=(l_{1},l_{2},...,l_{N})$, $l_{i}$ are arbitrary
positive non-integers. The norm in such spaces is defined by

\begin{equation}
\left\Vert u\right\Vert _{C^{\overline{l}}(R^{N})}\equiv \left\vert
u\right\vert _{R^{N}}^{(\overline{l})}=\left\vert u\right\vert
_{R^{N}}^{(0)}+\sum_{i=1}^{N}\left\langle u\right\rangle
_{x_{i},R^{N}}^{(l_{i})},  \label{1.3.01}
\end{equation}

\begin{equation}
\left\langle u\right\rangle _{x_{i},R^{N}}^{(l_{i})}=\sup_{x\in
R^{N},h>0}\frac{\left\vert
D_{x_{i}}^{[l_{i}]}u(x_{1},...,x_{i}+h,...,x_{N})-D_{x_{i}}^{[l_{i}]}u(x)%
\right\vert }{h^{l_{i}-[l_{i}]}},  \label{1.3.02}
\end{equation}
where $[l_{i}]$ is the integer part of the number $l_{i}$,
$D_{x_{i}}^{[l_{i}]}u$ is the derivative of order $[l_{i}]$ with
respect to the variable $x_{i}$ from a function $u$. Seminorm
(1.3.02) can be equivalently defined by (\cite{Triebel},\cite{Sol15}, \cite{Gol18} )

\begin{equation}
\left\langle u\right\rangle _{x_{i},R^{N}}^{(l_{i})}\simeq \sup_{x\in
R^{N},h>0}\frac{\left\vert \Delta _{h,x_{i}}^{k}u(x)\right\vert
}{h^{l_{i}}},\quad k>l_{i},  \label{1.3.03}
\end{equation}
where  $\Delta _{h,x_{i}}u(x)=u(x_{1},...,x_{i}+h,...,x_{N})-u(x)$
is the difference from a function $u(x)$ with respect to the
variable $x_{i}$ and with step $h$, $\Delta _{h,x_{i}}^{k}u(x)=$
$\Delta _{h,x_{i}}\left( \Delta _{h,x_{i}}^{k-1}u(x)\right) =\left(
\Delta _{h,x_{i}}\right) ^{k}u(x)$ is the difference of power $k$.
Note that functions from the space $C^{\overline{l}}(R^{N})$ have
also mixed derivatives up to definite orders and all derivatives are
H\"{o}lder continuous with respect to all variables with some
exponents in accordance with ratios between the exponents $l_{i}$.


Define further the space $\mathcal{H}^{\gamma \alpha }(R^{N})\subset
C^{\gamma \alpha }(R^{N})\cap L_{2}(R^{N})$ as the closure of
functions $u(x)$ from $C^{\gamma \alpha }(R^{N})$ with finite
supports in the norm

\begin{equation}
\left\Vert u\right\Vert _{\mathcal{H}^{\gamma \alpha }(R^{N})}\equiv
\left\Vert u\right\Vert _{L_{2}(R^{N})}+\sum_{i=1}^{N}\left\langle
u\right\rangle _{x_{i},R^{N}}^{(\gamma \alpha _{i})}.  \label{1.4}
\end{equation}

Analogously define the space $\mathcal{H}^{\overline{l}}(R^{N})$
with arbitrary positive non-integer $l_{i}$ with the norm

\begin{equation}
\left\Vert u\right\Vert _{\mathcal{H}^{\overline{l}}(R^{N})}\equiv
\left\Vert u\right\Vert _{L_{2}(R^{N})}+\sum_{i=1}^{N}\left\langle
u\right\rangle _{x_{i},R^{N}}^{(l_{i})}.  \label{1.4.01}
\end{equation}


It was shown in \cite{1} that $\left\vert u\right\vert
_{R^{N}}^{(0)}\leq С \left\Vert u\right\Vert _{\mathit{Н }^{\gamma
\alpha }(R^{N})}$ and so

\begin{equation}
\left\vert u\right\vert _{R^{N}}^{(\gamma \alpha )}\leq С \left\Vert
u\right\Vert _{\mathit{Н }^{\gamma \alpha }(R^{N})},  \label{1.5}
\end{equation}
where here and below we denote by  $C$, $\nu $ all absolute
constants or constant depending on fixed data only.

Let a function $\widetilde{m}(\xi )$, $\xi \in R^{N}$ be defined in
$R^{N}$ and let it be measurable and bounded. Define the operator
 $M:\mathcal{H}^{\gamma \alpha }(R^{N})\rightarrow
$ $L_{2}(R^{N})$ according to the formula

\begin{equation}
Mu\equiv \mathit{F}^{-1}(\widetilde{m}(\xi )\widetilde{u}(\xi )),
\label{1.6}
\end{equation}
where for a function $u(x)\in L_{1}(R^{N})$

\begin{equation}
\widetilde{u}(\xi )\equiv \mathit{F}(u)=\int\limits_{R^{N}}e^{-ix\xi
}u(x)dx \label{1.7.1.1}
\end{equation}
is Fourier transform of $u(x)$ and we extend Fourier transform on
the space $L_{2}(R^{N})$. Denote by
$\mathit{F}^{-1}\widetilde{u}(\xi )$ the inverse Fourier transform
of the function $\widetilde{u}(\xi )$.

Since $u(x)\in L_{2}(R^{N})$ and the function $\widetilde{m}(\xi )$
is bounded, the operator $M$ is well defined. We call the function
$\widetilde{m}(\xi )$ a Fourier multiplier.

Let the set of the variables $(\xi _{1},...,\xi _{N})=\xi $ is
represented as a union of $r$ subsets of length $N_{i}$,
$i=\overline{1,r} $ so that

\begin{equation*}
N_{1}+...+N_{r}=N,\quad \xi =(y_{1},...,y_{r}),\quad y_{1}=(\xi _{1},...,\xi
_{N_{1}}),...,y_{r}=(\xi _{N_{1}+...+N_{r-1}+1},...,\xi _{N}).
\end{equation*}

Let further $\omega _{i}$, $i=\overline{1,r}$, mean multi-indices of
length $N_{i}$

\begin{equation*}
\omega _{1}=(\omega _{1,1},...,\omega _{1,N_{1}}),...,\omega _{r}=(\omega
_{r,1},...,\omega _{r,N_{r}}),\quad \omega _{i,k}\in \mathbf{N\cup }\left\{
0\right\} ,
\end{equation*}
and thy symbol $D_{y_{i}}^{\omega _{i}}\widetilde{u}(\xi )$ means a
derivative of a function $\widetilde{u}(\xi )$ of order $\left\vert
\omega _{i}\right\vert =\omega _{i,1}+...+\omega _{i,N_{i}}$ with
respect to the group of variables $y_{i}=(\xi _{k_{1}},...,\xi
_{k_{N_{i}}})$, that is $D_{y_{i}}^{\omega _{i}}\widetilde{u}(\xi
)=D_{\xi _{k_{1}}}^{\omega _{i,1}}...D_{\xi _{k_{N_{i}}}}^{\omega
_{i,N_{i}}}\widetilde{u}(\xi )$. Let also $p\in (1,2]$ and positive
integers $s_{i}$, $i=\overline{1,r}$, satisfy the inequalities

\begin{equation}
s_{i}>\frac{N_{i}}{p},\quad i=\overline{1,r}.  \label{1.7}
\end{equation}

Denote for $\nu >0$

\begin{equation*}
B_{\nu }=\left\{ \xi \in R^{N}:\nu \leq \left\vert \xi \right\vert \leq \nu
^{-1}\right\} .
\end{equation*}

Suppose that for some $\nu >0$ the function $\widetilde{m}(\xi )$
satisfies with some $\mu
>0$ uniformly in $\lambda >0$ the condition

\begin{equation}
\sum_{\left\vert \omega _{i}\right\vert \leq s_{i}}\left\Vert
D_{y_{1}}^{\omega _{1}}D_{y_{2}}^{\omega _{2}}...D_{y_{r}}^{\omega
_{r}}\widetilde{m}(\lambda ^{\frac{1}{\alpha _{1}}}\xi _{1},...,\lambda
^{\frac{1}{\alpha _{N}}}\xi _{N})\right\Vert _{L_{p}(B_{\nu })}\leq \mu .  \label{1.8}
\end{equation}

\begin{theorem}
\label{T1.1} ( \cite{1} : T.2.1, L.2.1, L.2.2, T.2.2, T.2.3)

If the function $\widetilde{m}(\xi )$ satisfies condition
 \eqref{1.8} then the operator $M$, which is defined in
 \eqref{1.6}, is a linear bounded operator from the space
 $\mathcal{H}^{\gamma \alpha }(R^{N})$ to itself and

\begin{equation}
\left\Vert Mu\right\Vert _{\mathcal{H} ^{\gamma \alpha }(R^{N})}\leq
C(N,\gamma ,\alpha ,p,\nu ,s_{i})\mu \left\Vert u\right\Vert _{\mathcal{H}
^{\gamma \alpha }(R^{N})}.  \label{1.9}
\end{equation}
\end{theorem}
Note that very often condition \eqref{1.8} can be easily verified in
applications to differential equations. It is the case when the
function $\widetilde{m}(\xi )$ is anisotropic homogeneous of degree
zero, that is when  $\widetilde{m}(\lambda ^{\frac{1}{\alpha
_{1}}}\xi _{1},...,\lambda ^{\frac{1}{\alpha _{N}}}\xi
_{N})=\widetilde{m}(\xi )$. Note also that condition \eqref{1.8}
contains derivatives of the function $\widetilde{m}(\xi )$ wit
respect to variables $y_{i}$ only up to the order $s_{i}$. The case
 $r=1$, $N_{1}=N$, $p=2$ is contained in Lemma 2.1 in \cite{1} and
 Lemma 2.2 in \cite{1} contains the case $r=N$, $N_{i}=1$, $s_{i}=1$ .
The general case completely analogous (see the proofs of lemmas
\ref{L2.1}- \ref{L2.3} below).


The further content of the paper is as follows. In Section \ref{s2}
we prove a theorem about Fourier multipliers in spaces of functions
with H\"{o}lder condition with respect to only a part of
independent variables. In this section we also give comparatively
simple sufficient conditions for the theorem. As a conclusion of the
section we demonstrate applications of the theorem about Fourier
multipliers by two very simple but interesting in our opinion
examples for the Laplace equation and for the heat equation. In
section \ref{s3}, we apply the results of Section \ref{s2} to a
sketch of an investigation of model problems for a linearized
Cahn-Hilliard equation with dynamic boundary conditions of two
types.

More detailed investigation of initial-boundary value problems for
Cahn-Hilliard equation with dynamic boundary conditions will be
given in a forthcoming paper.


\section{Theorems about Fourier multipliers in spaces of functions with H\"{o}lder condition
with respect to a part of their variables }

\label{s2}

In this section we prove a theorem about Fourier multipliers, which
is a generalization of Theorem 2.1 from \cite{1}. The schema of the
proof is a modification of corresponding schemas from  \cite{1},
\cite{Herm}.

Define an anisotropic "distance" $\rho $\ in space $R^{N}$ between
points $x$ and $y$ according to the formula

\begin{equation}
\rho (x-y)=\sum\limits_{k=1}^{N-2}\left\vert x_{k}-y_{k}\right\vert
+\left\vert x_{N-1}-y_{N-1}\right\vert ^{\alpha }+\left\vert
x_{N}-y_{N}\right\vert ^{\beta },\quad \alpha \in (0,1],\beta \gtrdot 0.
\label{A.0}
\end{equation}

Let us stress that the exponent $\beta \gtrdot 0$ is an arbitrary
great positive number.

Choose a function $\omega (\rho ):[0,+\infty )\rightarrow \lbrack
0,1]$ from the class  $С ^{\infty }$ such that $\omega \equiv 1$ on the
interval $[1/2,2]$ and $\omega \equiv 0$ on the set $[0,1/4]\cup
\lbrack 4,+\infty )$. Denote

\begin{equation*}
\chi :R^{N}\rightarrow \lbrack 0,1],\quad \chi (\xi )\equiv \omega (\rho
(\xi )),\quad \xi \in R^{N}.
\end{equation*}

Let a function  $\widetilde{m}(\xi )\in C(R^{N}\backslash \{0\})$ be
bounded. For $x\in R^{N}$  and for an integer $j\in Z$ denote

\begin{equation}
A_{j}x\equiv (2^{j}x_{1},...,2^{j}x_{N-2},2^{\frac{j}{\alpha
}}x_{N-1},2^{\frac{j}{\beta }}x_{N}),\quad a_{j}=\det A_{j}=2^{j(N-2)+\frac{j}{\alpha
}+\frac{j}{\beta }}.  \label{A.8.1}
\end{equation}

Denote further $\widetilde{m}_{j}(\xi )=\widetilde{m}(\xi )\chi
(A_{j}^{-1}\xi )$, denote by $m_{j}(x)$ the inverse Fourier
transform of the function  $\widetilde{m}_{j}(\xi )$, and denote

\begin{equation}
n_{j}(x)=a_{j}^{-1}m_{j}(A_{j}^{-1}x).  \label{A.8.2}
\end{equation}

For convenience we also denote for $x\in R^{N}$ the variables
$x^{\prime }=(x_{1},...,x_{N-2},x_{N-1})$, $x^{\prime \prime
}=(x_{1},...,x_{N-2})$.

Let with some $\mu >0$ the following conditions are satisfied

\begin{equation}
\widetilde{m}(\xi )|_{\xi ^{\prime }=0}=\widetilde{m}(0,...,0,\xi
_{N})\equiv 0,\quad \xi _{N}\in R^{1},  \label{A.01}
\end{equation}

\begin{equation}
\int\limits_{R^{N}}(1+|x^{\prime \prime }|^{\gamma }+|x_{N-1}|^{\alpha
\gamma })|n_{j}(x)|dx\leq \mu ,\quad j\in Z.  \label{A.02}
\end{equation}

Let finally we have a function  $u(x)\in C_{x^{\prime \prime
},x_{N-1}}^{\gamma ,\alpha \gamma }(R^{N})$ with finite support,
that is particularly

\begin{equation*}
\left\langle u\right\rangle _{x^{\prime \prime },R^{N}}^{(\gamma
)}+\left\langle u\right\rangle _{x_{N-1},R^{N}}^{(\alpha \gamma )}<\infty .
\end{equation*}

Denote

\begin{equation*}
v(x)=m(x)\ast u(x)\equiv F^{-1}(\widetilde{m}(\xi )\widetilde{u}(\xi )).
\end{equation*}

\begin{theorem}
\label{T2.1}

Under conditions  \eqref{A.01}, \eqref{A.02} the function $v(x)$
satisfies H\"{o}lder conditions with respect to all variables
and the following estimate is valid

\begin{equation}
\left\langle v\right\rangle _{x^{\prime \prime },R^{N}}^{(\gamma
)}+\left\langle v\right\rangle _{x_{N-1},R^{N}}^{(\alpha \gamma
)}+\left\langle v\right\rangle _{x_{N},R^{N}}^{(\beta \gamma )}\leq
C\mu \left( \left\langle u\right\rangle _{x^{\prime \prime
},R^{N}}^{(\gamma )}+\left\langle u\right\rangle
_{x_{N-1},R^{N}}^{(\alpha \gamma )}\right) . \label{A.03}
\end{equation}
\end{theorem}

\begin{proof}
We will try to retain the notation of \cite{1} where it is possible.

Let $\psi \in C^{\infty }([0,\infty ))$, $0\leq \psi \leq 1$, $\psi
\equiv 1$ on $[0,1]$ \ and $\psi \equiv 0$ on $[2,\infty )$. \
Denote $\varphi (\rho )=\psi (\rho )-\psi (2\rho )$ for $\rho \in
\lbrack 0,\infty )$. The function $\varphi (\rho )$ possess the
properties: $\varphi (\rho )\equiv 0$ on $[0,1/2]$ and $\varphi
(\rho )\equiv 0$ on $[2,\infty )$. Denote further

\begin{equation*}
\varphi _{j}(\rho )=\varphi (\frac{\rho }{2^{j}}),\quad \varphi
_{j}:[0,\infty )\rightarrow \lbrack 0,1],\quad j\in Z.
\end{equation*}

This set of the functions satisfies by the definition

\begin{equation}
\sum\limits_{j=-\infty }^{\infty }\varphi _{j}(\rho )=1,\quad \rho \in
(0,\infty ).  \label{A.1}
\end{equation}

Define functions  $\widetilde{\Phi }$ and $\widetilde{\Phi
}_{j}:R^{N}\rightarrow \lbrack 0,1]$ \ according to the formulas
($\rho $ is from  \eqref{A.0})

\begin{equation}
\widetilde{\Phi }(\xi )\equiv \varphi \circ \rho (\xi )=\varphi (\rho (\xi
)),  \label{A.2}
\end{equation}

\begin{equation}
\widetilde{\Phi }_{j}(\xi )\equiv \varphi _{j}\circ \rho (\xi )=\varphi
_{j}(\rho (\xi ))=\varphi \left( \frac{\rho (\xi )}{2^{j}}\right) =.
\label{A.3}
\end{equation}

\begin{equation}
=\varphi \left( \rho \left( \frac{\xi ^{\prime \prime }}{2^{j}},\frac{\xi
_{N-1}}{2^{j/\alpha }},\frac{\xi _{N}}{2^{j/\beta }}\right) \right)
=\widetilde{\Phi }\left( \frac{\xi ^{\prime \prime }}{2^{j}},\frac{\xi
_{N-1}}{2^{j/\alpha }},\frac{\xi _{N}}{2^{j/\beta }}\right) .  \label{A.4}
\end{equation}

By the definition of the functions $\widetilde{\Phi }_{j}$ and in
view of  \eqref{A.1}

\begin{equation}
\sum\limits_{j=-\infty }^{\infty }\widetilde{\Phi }_{j}(\xi )\equiv 1,\quad
\xi \in R^{N}\backslash \{0\}.  \label{A.5}
\end{equation}

We use this equality to represent the function  $\widetilde{u}(\xi
)=F(u(x))$ as

\begin{equation}
\widetilde{u}(\xi )=\sum\limits_{j=-\infty }^{\infty }\widetilde{u}_{j}(\xi
),\quad \widetilde{u}_{j}(\xi )=\widetilde{u}(\xi )\widetilde{\Phi }_{j}(\xi
),\quad \xi \in R^{N}\backslash \{0\}.  \label{A.6}
\end{equation}

Denote also

\begin{equation}
\Phi (x)=F^{-1}(\widetilde{\Phi }(\xi )),\quad \Phi
_{j}(x)=F^{-1}(\widetilde{\Phi }_{j}(\xi )).  \label{A.7}
\end{equation}

In view of \eqref{A.4}

\begin{equation*}
\Phi _{j}(x)=F^{-1}(\widetilde{\Phi }_{j}(\xi ))=\int\limits_{R^{N}}e^{ix\xi
}\widetilde{\Phi }\left( \frac{\xi ^{\prime \prime }}{2^{j}},\frac{\xi
_{N-1}}{2^{j/\alpha }},\frac{\xi _{N}}{2^{j/\beta }}\right) d\xi .
\end{equation*}

Making in the last integral the change of the variables  $\xi
^{\prime \prime }=2^{j}\eta ^{\prime \prime }$, $\xi
_{N-1}=2^{j/\alpha }\eta _{N-1}$, $\xi _{N}=2^{j/\beta }\eta _{N}$,
we obtain

\begin{equation}
\Phi _{j}(x)=2^{(N-2)j+\frac{j}{\alpha }+\frac{j}{\beta }}\Phi \left(
2^{j}x^{\prime \prime },2^{j/\alpha }x_{N-1},2^{j/\beta }x_{N}\right)
=a_{j}\Phi (A_{j}x).  \label{A.8}
\end{equation}

In view of the above definition of the function $\chi (\xi ) $ and
in view of the definition of the functions  $\widetilde{\Phi }(\xi
)$ и $\widetilde{\Phi }_{j}(\xi )$ we have for all $\xi \in R^{N}$

\begin{equation*}
\widetilde{\Phi }(\xi )=\widetilde{\Phi }(\xi )\chi (\xi ),\quad
\widetilde{\Phi }_{j}(\xi )=\widetilde{\Phi }_{j}(\xi )\chi \left( \frac{\xi ^{\prime
\prime }}{2^{j}},\frac{\xi _{N-1}}{2^{j/\alpha }},\frac{\xi _{N}}{2^{j/\beta
}}\right) .
\end{equation*}

Consequently,

\begin{equation}
\widetilde{u}_{j}(\xi )=\widetilde{u}(\xi )\widetilde{\Phi }_{j}(\xi
)=\widetilde{u}_{j}(\xi )\chi \left( \frac{\xi ^{\prime \prime
}}{2^{j}},\frac{\xi _{N-1}}{2^{j/\alpha }},\frac{\xi _{N}}{2^{j/\beta }}\right) \equiv
\widetilde{u}_{j}(\xi )\chi _{j}(\xi ),  \label{A.9}
\end{equation}

where

\begin{equation}
\chi _{j}(\xi )\equiv \chi \left( \frac{\xi ^{\prime \prime
}}{2^{j}},\frac{\xi _{N-1}}{2^{j/\alpha }},\frac{\xi _{N}}{2^{j/\beta }}\right) .
\label{A.10}
\end{equation}

Denote in the cense of distributions

\begin{equation}
v(x)=m(x)\ast u(x),  \label{A.11}
\end{equation}

that is

\begin{equation*}
\widetilde{v}(\xi )=\widetilde{m}(\xi )\widetilde{u}(\xi
)=\sum\limits_{j=-\infty }^{\infty }\widetilde{m}(\xi )\widetilde{u}_{j}(\xi
)=
\end{equation*}

\begin{equation}
=\sum\limits_{j=-\infty }^{\infty }\widetilde{m}(\xi )\chi _{j}(\xi
)\widetilde{u}_{j}(\xi )\equiv \sum\limits_{j=-\infty }^{\infty
}\widetilde{m}_{j}(\xi )\widetilde{u}_{j}(\xi )\equiv \sum\limits_{j=-\infty }^{\infty
}\widetilde{v}_{j}(\xi ),  \label{A.12}
\end{equation}

where

\begin{equation}
\widetilde{m}_{j}(\xi )\equiv \widetilde{m}(\xi )\chi _{j}(\xi ),\quad
\widetilde{v}_{j}(\xi )=\widetilde{m}_{j}(\xi )\widetilde{u}_{j}(\xi ).
\label{A.12+1}
\end{equation}

Consider the function $v_{j}(x)=F^{-1}(\widetilde{v}_{j}(\xi ))$.
Using it's definition and \eqref{A.8.1}, we represent this function
in the form

\begin{equation*}
v_{j}(x)=u(x)\ast \Phi _{j}(x)\ast
m_{j}(x)=\int\limits_{R^{N}}u(x-y)dy\int\limits_{R^{N}}m_{j}(y-z)\Phi
_{j}(z)dz=
\end{equation*}

\begin{equation}
=\int\limits_{R^{N}}u(x-y)dy\int\limits_{R^{N}}m_{j}(y-z)a_{j}\Phi
(A_{j}z)dz.  \label{A.12+1.1}
\end{equation}

Making in the last integral the change of the variables $k=A_{j}z$,
$dk=a_{j}dz$, we obtain

\begin{equation*}
v_{j}(x)=\int\limits_{R^{N}}u(x-y)dy\int\limits_{R^{N}}m_{j}(y-A_{j}^{-1}k)%
\Phi (k)dk.
\end{equation*}

Making now  the change of the variables $y=A_{j}^{-1}z$, we arrive
at the expression

\begin{equation}
v_{j}(x)=\int\limits_{R^{N}}u(x-A_{j}^{-1}z)dz\int\limits_{R^{N}}\left[
a_{j}^{-1}m_{j}(A_{j}^{-1}(z-k))\right] \Phi (k)dk\equiv
\int\limits_{R^{N}}u(x-A_{j}^{-1}z)\theta _{j}(z)dz,  \label{A.12+2}
\end{equation}

where

\begin{equation}
\theta _{j}(z)=n_{j}(z)\ast \Phi (z)=\int\limits_{R^{N}}n_{j}(z-k)\Phi (k)dk,
\label{A.15}
\end{equation}

and the function $n_{j}(z)$ is defined in \eqref{A.8.2}.

Calculate now the derivatives $v_{x_{i}}$, $i=1,2,...,N$. For this
we use properties of a convolution and analogously \eqref{A.12+1.1}
represent $v_{j}(x)$ as

\begin{equation*}
v_{j}(x)=\int\limits_{R^{N}}a_{j}\Phi
(A_{j}(x-y))dy\int\limits_{R^{N}}u(y-z)m_{j}(z)dz.
\end{equation*}

Let first $i=N-1$. Then

\begin{equation*}
\left( v_{j}(x)\right) _{x_{N-1}}=2^{\frac{j}{\alpha
}}\int\limits_{R^{N}}a_{j}\Phi
^{(i)}(A_{j}(x-y))dy\int\limits_{R^{N}}u(y-z)m_{j}(z)dz,
\end{equation*}

where

\begin{equation}
\Phi ^{(i)}(z)\equiv \frac{\partial \Phi }{\partial z_{i}}(z),\quad
i=1,2,...,N.  \label{A.16.1}
\end{equation}

Using properties of a convolution and making a series of changes of
variables with the matrices  $A_{j}$ and $A_{j}^{-1}$, we obtain
completely analogously to \eqref{A.12+2}

\begin{equation}
\left( v_{j}(x)\right) _{x_{N-1}}=2^{\frac{j}{\alpha
}}\int\limits_{R^{N}}u(x-A_{j}^{-1}z)\left( \theta _{j}\right) _{z_{N-1}}(z)dz,
\label{A.17}
\end{equation}

where

\begin{equation*}
\left( \theta _{j}\right) _{z_{N-1}}(z)=\int\limits_{R^{N}}n_{j}(z-k)\Phi
_{k_{N-1}}(k)dk.
\end{equation*}

Completely analogously for $1\leq i\leq N-2$

\begin{equation}
\left( v_{j}(x)\right)
_{x_{i}}=2^{j}\int\limits_{R^{N}}u(x-A_{j}^{-1}z)\left( \theta _{j}\right)
_{z_{i}}(z)dz,  \label{A.18}
\end{equation}

and for $i=N$

\begin{equation*}
\left( v_{j}(x)\right) _{x_{N}}=2^{\frac{j}{\beta
}}\int\limits_{R^{N}}u(x-A_{j}^{-1}z)\left( \theta _{j}\right) _{z_{N}}(z)dz,
\end{equation*}

and also in more general case

\begin{equation*}
D_{x_{N}}^{k}\left( v_{j}(x)\right) =2^{\frac{kj}{\beta
}}\int\limits_{R^{N}}u(x-A_{j}^{-1}z)D_{z_{N}}^{k}\left( \theta _{j}\right)
(z)dz,\quad k=0,1,2....
\end{equation*}

Now note that for almost all $z_{N}$

\begin{equation*}
f_{j}(z_{N})\equiv \int\limits_{R^{N-1}}\theta _{j}(z^{\prime
},z_{N})dz^{\prime }\equiv 0,\quad f_{j}^{(i)}(z_{N})\equiv
\int\limits_{R^{N-1}}\left( \theta _{j}\right) _{z_{i}}(z^{\prime
},z_{N})dz^{\prime }\equiv 0,\quad
\end{equation*}

\begin{equation}
\int\limits_{R^{N-1}}D_{z_{N}}^{k}\left( \theta _{j}\right) (z^{\prime
},z_{N})dz^{\prime }\equiv 0.  \label{A.19}
\end{equation}

We show the first of these relations as the second and the third are
completely similar. It suffices to show that the Fourier transform
of $f_{j}(z_{N})$ with respect to $z_{N}$ is identically equal to
zero

\begin{equation*}
F_{N}(f_{j})=\widetilde{f}_{j}(\xi _{N})=\int\limits_{R^{1}}e^{-iz_{N}\xi
}dz_{N}\int\limits_{R^{N-1}}\theta _{j}(z^{\prime },z_{N})dz^{\prime }=
\end{equation*}

\begin{equation*}
=\int\limits_{R^{N-1}}dz^{\prime }\int\limits_{R^{1}}e^{-iz_{N}\xi }\theta
_{j}(z^{\prime },z_{N})dz_{N}=\int\limits_{R^{N-1}}F_{N}(\theta
_{j})(z^{\prime },\xi _{N})dz^{\prime }.
\end{equation*}

Since the integral with respect to $z^{\prime } $ of the function
$F_{N}(\theta _{j})(z^{\prime },\xi _{N})$ is equal to the value at
$\xi ^{\prime }=0$ of it's Fourier transform with respect to the
same variables $z^{\prime }$, then

\begin{equation*}
\int\limits_{R^{N-1}}F_{N}(\theta _{j})(z^{\prime },\xi _{N})dz^{\prime
}=\left. \left[ \int\limits_{R^{N-1}}e^{-iz^{\prime }\xi ^{\prime
}}F_{N}(\theta _{j})(z^{\prime },\xi _{N})dz^{\prime }\right] \right\vert
_{\xi ^{\prime }=0}=
\end{equation*}

\begin{equation*}
=\left. \widetilde{\theta }_{j}(\xi ^{\prime },\xi _{N})\right\vert _{\xi
^{\prime }=0}=\widetilde{\theta }_{j}(0,\xi _{N})\equiv 0.
\end{equation*}
The last identity follows from the fact that

\begin{equation}
\widetilde{n}_{j}(\xi )=\int\limits_{R^{N}}e^{-ix\xi
}a_{j}^{-1}m_{j}(A_{j}^{-1}x)dx=\widetilde{m}_{j}(A_{j}\xi ),  \label{A.20}
\end{equation}

and consequently

\begin{equation*}
\widetilde{\theta }_{j}(\xi )=\widetilde{n}_{j}(\xi )\widetilde{\Phi
}(\xi )=\widetilde{m}_{j}(A_{j}\xi )\widetilde{\Phi }(\xi
)=\widetilde{m}(A_{j}\xi )\chi _{j}(\xi )\widetilde{\Phi }(\xi ).
\end{equation*}
Therefore in view of \eqref{A.01} we have $\widetilde{\theta
}_{j}(0,\xi _{N})\equiv 0$. Thus the first relation \eqref{A.19} is
proved. The second  is similar.

Let us obtain now the estimates

\begin{equation}
\int\limits_{R^{N}}|z^{\prime \prime }|^{\gamma }\left\vert \theta
_{j}(z)\right\vert dz\leq C\mu ,\quad \int\limits_{R^{N}}|z_{N-1}|^{\alpha
\gamma }\left\vert \theta _{j}(z)\right\vert dz\leq C\mu ,\quad j\in Z,
\label{AA.1}
\end{equation}

\begin{equation}
\int\limits_{R^{N}}|z^{\prime \prime }|^{\gamma }\left\vert \left( \theta
_{j}\right) _{z_{k}}(z)\right\vert dz\leq C\mu ,\quad
\int\limits_{R^{N}}|z_{N-1}|^{\alpha \gamma }\left\vert \left( \theta
_{j}\right) _{z_{k}}(z)\right\vert dz\leq C\mu ,\quad j\in
Z,k=\overline{1,N},  \label{AA.2}
\end{equation}

\begin{equation}
\int\limits_{R^{N}}|z^{\prime \prime }|^{\gamma }\left\vert
D_{z_{N}}^{n}\left( \theta _{j}\right) (z)\right\vert dz\leq C\mu ,\quad
\int\limits_{R^{N}}|z_{N-1}|^{\alpha \gamma }\left\vert D_{z_{N}}^{n}\left(
\theta _{j}\right) (z)\right\vert dz\leq C\mu ,\quad j\in Z,n=1,2,...,
\label{AA.2.01}
\end{equation}
where $\mu $ is from condition \eqref{A.02}. We obtain only the
first inequality \eqref{AA.1} because the rest is quite similar.
Indeed, if  $y^{\prime \prime }\in R^{N-2}$ then we use the
inequality $|z^{\prime \prime }|^{\gamma }\leq C\left( |y^{\prime
\prime }|^{\gamma }+|z^{\prime \prime }-y^{\prime \prime }|^{\gamma
}\right) $ and in view of the definition of $\theta _{j}$ in
\eqref{A.15} we obtain

\begin{equation*}
\int\limits_{R^{N}}|z^{\prime \prime }|^{\gamma }\left\vert \theta
_{j}(z)\right\vert dz\leq \int\limits_{R^{N}}|z^{\prime \prime }|^{\gamma
}dz\int\limits_{R^{N}}|n_{j}(y)||\Phi (z-y)|dy\leq
\end{equation*}

\begin{equation*}
\leq C\int\limits_{R^{N}}dz\int\limits_{R^{N}}\left[ |y^{\prime \prime
}|^{\gamma }|n_{j}(y)|\right] |\Phi
(z-y)|dy+C\int\limits_{R^{N}}dz\int\limits_{R^{N}}|n_{j}(y)|\left[
|z^{\prime \prime }-y^{\prime \prime }|^{\gamma }|\Phi (z-y)|\right] dy\leq
\end{equation*}

\begin{equation*}
\leq \int\limits_{R^{N}}\left( 1+|y^{\prime \prime }|^{\gamma }\right)
|n_{j}(y)|dy\leq C\mu ,
\end{equation*}
as it follows from properties of the function $\Phi (x)$.

Let us now estimate the H\"{o}lder constant of the function
$v(x)$. For this we estimate $|v_{j}(x)|$ and
$|(v_{j})_{x_{i}}(x)|$. From  \eqref{A.12+2} and \eqref{A.19} it
follows that

\begin{equation*}
v_{j}(x)=\int\limits_{R^{1}}dz_{N}\int\limits_{R^{N-1}}u(x-A_{j}^{-1}z)%
\theta _{j}(z)dz^{\prime }=
\end{equation*}

\begin{equation*}
=\int\limits_{R^{1}}dz_{N}\int\limits_{R^{N-1}}\left[
u(x-A_{j}^{-1}z)-u(x^{\prime \prime },x_{N-1},x_{N}-\frac{z_{N}}{2^{j/\beta }})\right] \theta
_{j}(z)dz^{\prime }=
\end{equation*}

\begin{equation}
=\int\limits_{R^{N}}\left[ u(x^{\prime \prime }-\frac{z^{\prime \prime
}}{2^{j}},x_{N-1}-\frac{z_{N-1}}{2^{j/\alpha }},x_{N}-\frac{z_{N}}{2^{j/\beta
}})-u(x^{\prime \prime },x_{N-1},x_{N}-\frac{z_{N}}{2^{j/\beta }})\right]
\theta _{j}(z)dz.  \label{A.21}
\end{equation}

From this, in view of the fact that $u(x)$ satisfies the
H\"{o}lder condition with respect to the first $N-1$ variables
and in view of estimate \eqref{AA.1}, it follows that

\begin{equation}
|v_{j}(x)|\leq C2^{-\gamma j}\left\langle u\right\rangle _{x^{\prime
}}^{(\gamma ,\alpha \gamma )}\int\limits_{R^{N}}\left( |z^{\prime \prime
}|^{\gamma }+|z_{N-1}|^{\alpha \gamma }\right) \left\vert \theta
_{j}(z)\right\vert dz\leq C\mu \left\langle u\right\rangle _{x^{\prime
}}^{(\gamma ,\alpha \gamma )}2^{-\gamma j}.  \label{A.22}
\end{equation}

And for $i=\overline{1,N-2}$ we similarly obtain

\begin{equation*}
(v_{j})_{x_{i}}(x)=2^{j}\int\limits_{R^{N}}\left[ u(x^{\prime \prime
}-\frac{z^{\prime \prime }}{2^{j}},x_{N-1}-\frac{z_{N-1}}{2^{j/\alpha
}},x_{N}-\frac{z_{N}}{2^{j/\beta }})-u(x^{\prime \prime
},x_{N-1},x_{N}-\frac{z_{N}}{2^{j/\beta }})\right] \left( \theta _{j}\right) _{z_{i}}(z)dz,
\end{equation*}
and therefore we have similarly to \eqref{A.22}

\begin{equation}
|(v_{j})_{x_{i}}(x)|\leq C\mu \left\langle u\right\rangle _{x^{\prime
}}^{(\gamma ,\alpha \gamma )}2^{j-\gamma j},\quad i=\overline{1,N-2}.
\label{A.23}
\end{equation}

Likewise for $i=N-1$ and for $i=N$

\begin{equation}
|(v_{j})_{x_{N-1}}(x)|\leq C\mu \left\langle u\right\rangle _{x^{\prime
}}^{(\gamma ,\alpha \gamma )}2^{\frac{j}{\alpha }-\gamma j},  \label{A.24}
\end{equation}

\begin{equation}
|(v_{j})_{x_{N}}(x)|\leq C\mu \left\langle u\right\rangle _{x^{\prime
}}^{(\gamma ,\alpha \gamma )}2^{\frac{j}{\beta }-\gamma j},  \label{A.24.1}
\end{equation}

and more generally for $i=N$

\begin{equation}
|D_{x_{N}}^{k}(v_{j})(x)|\leq C\mu \left\langle u\right\rangle _{x^{\prime
}}^{(\gamma ,\alpha \gamma )}2^{\frac{kj}{\beta }-\gamma j},\quad k=1,2,....
\label{A.24.02}
\end{equation}

Let now $x.y\in R^{N}$. Consider first the case when $\beta \gamma
<1$. We have

\begin{equation}
|v(x)-v(y)|\leq \sum\limits_{j=-\infty }^{\infty
}|v_{j}(x)-v_{j}(y)|=\sum\limits_{j\geq
n_{0}}|v_{j}(x)-v_{j}(y)|+\sum\limits_{j\leq n_{0}}|v_{j}(x)-v_{j}(y)|\equiv
S_{1}+S_{2},  \label{A.25}
\end{equation}
where  $n_{0}=-\log _{2}\rho (x-y)$. To estimate  $S_{1}$ we use
inequalities  \eqref{A.22}

\begin{equation*}
S_{1}\leq \sum\limits_{j\geq n_{0}}2|v_{j}|^{(0)}\leq C\mu \left\langle
u\right\rangle _{x^{\prime }}^{(\gamma ,\alpha \gamma )}\sum\limits_{j\geq
n_{0}}2^{-\gamma j}\leq
\end{equation*}

\begin{equation}
\leq C\mu \left\langle u\right\rangle _{x^{\prime }}^{(\gamma ,\alpha \gamma
)}2^{-n_{0}\gamma }\sum\limits_{k=0}^{\infty }2^{-\gamma k}\leq C\mu
\left\langle u\right\rangle _{x^{\prime }}^{(\gamma ,\alpha \gamma )}\rho
^{\gamma }(x-y).  \label{A.26}
\end{equation}

To estimate $S_{2}$ we use the mean value theorem for the difference
$|v_{j}(x)-v_{j}(y)|$ and estimates \eqref{A.23}- \eqref{A.24.1} for
the corresponding derivatives

\begin{equation*}
S_{2}\leq C\sum\limits_{j\leq n_{0}}\left(
\sum\limits_{k=1}^{N}|x_{k}-y_{k}||(v_{j})_{x_{k}}|^{(0)}\right) \leq
\end{equation*}

\begin{equation*}
\leq C\mu \left\langle u\right\rangle _{x^{\prime }}^{(\gamma ,\alpha \gamma
)}\sum\limits_{j\leq n_{0}}\left( |x^{\prime \prime }-y^{\prime \prime
}|2^{j-j\gamma }+|x_{N-1}-y_{N-1}|2^{\frac{j}{\alpha }-j\gamma
}+|x_{N}-y_{N}|2^{\frac{j}{\beta }-j\gamma }\right) \leq
\end{equation*}

\begin{equation*}
\leq C\mu \left\langle u\right\rangle _{x^{\prime }}^{(\gamma ,\alpha \gamma
)}\left( |x^{\prime \prime }-y^{\prime \prime }|2^{(1-\gamma
)n_{0}}\sum\limits_{k=0}^{\infty }2^{-(1-\gamma )k}\right. +
\end{equation*}

\begin{equation*}
+\left. |x_{N-1}-y_{N-1}|2^{n_{0}\left( \frac{1}{\alpha }-\gamma \right)
}\sum\limits_{k=0}^{\infty }2^{-(\frac{1}{\alpha }-\gamma
)k}+|x_{N}-y_{N}|2^{n_{0}\left( \frac{1}{\beta }-\gamma \right)
}\sum\limits_{k=0}^{\infty }2^{-(\frac{1}{\beta }-\gamma )k}\right) \leq
\end{equation*}

\begin{equation}
\leq C\mu \left\langle u\right\rangle _{x^{\prime }}^{(\gamma ,\alpha \gamma
)}\left( |x^{\prime \prime }-y^{\prime \prime }|\rho ^{-1+\gamma }(x-y)+
\right.  \label{A.27}
\end{equation}

\begin{equation*}
\left. +|x_{N-1}-y_{N-1}|\rho ^{-\frac{1}{\alpha }+\gamma
}(x-y)+|x_{N}-y_{N}|\rho ^{-\frac{1}{\beta }+\gamma }(x-y)\right) \leq
\end{equation*}

\begin{equation*}
\leq C\mu \left\langle u\right\rangle _{x^{\prime }}^{(\gamma ,\alpha \gamma
)}\rho ^{\gamma }(x-y),
\end{equation*}

where we have used the fact that $\beta \gamma <1$ and consequently
 $(\frac{1}{\beta }-\gamma )>0$.

From \eqref{A.26} and \eqref{A.27} it follows that

\begin{equation}
|v(x)-v(y)|\leq C\mu \left\langle u\right\rangle _{x^{\prime }}^{(\gamma
,\alpha \gamma )}\rho ^{\gamma }(x-y).  \label{A.27.01}
\end{equation}

This proves the theorem in the case $\beta \gamma <1$.

Let now $\beta \gamma >1$.

The proof in this case requires only a small change. Firstly,
selecting in the previous proof point $x$ and $y$ such that
$x_{N}=y_{N}$, that is considering the H\"{o}lder property of
the function $v(x)$  only with respect to the variables $x^{\prime
}$, we obtain estimate \eqref{A.27} with $|x_{N}-y_{N}|=0$. This
proves \eqref{A.27.01} for such $x$ and $y$ and hence this gives the
desired smoothness of $v(x)$ with respect to the variables
$x^{\prime }$. Now the smoothness property of this function in the
variable $x_{N}$ should be considered separately. For this purpose,
with definition \eqref{1.3.03} in mind,  we need to consider $k$-th
difference in variable $x_{N}$ of the function $v(x)$. Let $k$ be a
sufficiently large positive integer such that $k /\beta>\gamma$,
$h>0$. Similarly to the previous

\begin{equation*}
|\Delta _{h,x_{N}}^{k}v(x)|\leq \sum\limits_{j=-\infty }^{\infty }|\Delta
_{h,x_{N}}^{k}v_{j}(x)|=\sum\limits_{j\geq n_{0}}|\Delta
_{h,x_{N}}^{k}v_{j}(x)|+\sum\limits_{j\leq n_{0}}|\Delta
_{h,x_{N}}^{k}v_{j}(x)|\equiv S_{1}+S_{2},
\end{equation*}

where $n_{0}=-\log _{2}h^{\beta }$. The sum $S_{1}$ is estimated at
exactly the same way as above, which gives

\begin{equation*}
S_{1}\leq C\mu \left\langle u\right\rangle _{x^{\prime }}^{(\gamma ,\alpha
\gamma )}h^{\beta \gamma }.
\end{equation*}

The sum $S_{2}$ is also evaluated as before taking into account the
fact that

\begin{equation*}
|\Delta _{h,x_{N}}^{k}v_{j}(x)|\leq Ch^{k}\left\vert
D_{x_{N}}^{k}v_{j}(x)\right\vert ^{(0)}\leq C\mu \left\langle u\right\rangle
_{x^{\prime }}^{(\gamma ,\alpha \gamma )}2^{\frac{kj}{\beta }-\gamma j},
\end{equation*}

where we also used estimate \eqref{A.24.02}. This gives

\begin{equation*}
S_{2}\leq C\mu \left\langle u\right\rangle _{x^{\prime }}^{(\gamma ,\alpha
\gamma )}h^{k}2^{n_{0}\left( \frac{k}{\beta }-\gamma \right)
}\sum\limits_{k=0}^{\infty }2^{-(\frac{k}{\beta }-\gamma )k}\leq C\mu
\left\langle u\right\rangle _{x^{\prime }}^{(\gamma ,\alpha \gamma
)}h^{\beta \gamma }.
\end{equation*}

From the estimates for $S_{1}$ and $S_{2}$ it follows that

\begin{equation*}
|\Delta _{h,x_{N}}^{k}v(x)|\leq C\mu \left\langle u\right\rangle _{x^{\prime
}}^{(\gamma ,\alpha \gamma )}h^{\beta \gamma }.
\end{equation*}

Thus by definition \eqref{1.3.03} the theorem is proved.
\end{proof}

Following the idea of \cite{1} and similar to the conditions of
Theorem \ref{T1.1}, we give simple sufficient conditions on
$\widetilde{m}(\xi)$ to have condition \eqref{A.02}. Note fist that
after the change of the variables $y=A_{j}^{-1}x$ we obtain

\begin{equation}
\widetilde{n}_{j}(\xi )=C\int\limits_{R^{N}}e^{ix\xi
}a_{j}^{-1}m_{j}(A_{j}^{-1}x)dx=C\int\limits_{R^{N}}e^{i(y,A_{j}\xi
)}m_{j}(y)dy=\widetilde{m}_{j}(A_{j}\xi )=\widetilde{m}(A_{j}\xi )\chi (\xi
).  \label{B.1}
\end{equation}
Denote for $\lambda >0$ \ $A_{\lambda }\xi =(\lambda \xi ^{\prime
\prime },\lambda ^{\frac{1}{\alpha }}\xi _{N-1},\lambda
^{\frac{1}{\beta }}\xi _{N}) $ and denote $B_{0}=\left\{ \xi \in
R^{N}:1/8\leq \rho (\xi )\leq 8\right\} $. All sufficient conditions
to have \eqref{A.02}, which we state below, are linked with the
property of the Fourier transform

\begin{equation*}
-ix_{k}f(x)=\widetilde{f_{\xi _{k}}},
\end{equation*}
as well as with the well-known the Hausdorff-Young  inequality

\begin{equation}
\left\Vert f(x)\right\Vert _{L_{p^{\prime }}(R^{N})}\leq C_{N,p}\left\Vert
\widetilde{f}(\xi )\right\Vert _{L_{p}(R^{N})},\quad p\in (1,2],\quad
p^{\prime }=\frac{p}{p-1}.  \label{B.2}
\end{equation}

\begin{lemma}
\label{L2.1} Let uniformly in $\lambda >0$

\begin{equation*}
\widetilde{m}(A_{\lambda }\xi )\in W_{p}^{s}(B_{0}),\quad p\in (1,2],\quad
s>\frac{N}{p}+\gamma .
\end{equation*}

Then conditions \eqref{A.02} are satisfied and

\begin{equation}
\mu \leq \sup\limits_{\lambda }C\left\Vert \widetilde{m}(A_{\lambda }\xi
)\right\Vert _{W_{p}^{s}(B_{0})}.  \label{(B.3)}
\end{equation}
\end{lemma}

\begin{proof}
(compare \cite{1}).

In view of \eqref{B.1} for $r>N/p $

\begin{equation*}
\int\limits_{R^{N}}(1+\left\vert x^{\prime \prime }\right\vert ^{\gamma
}+\left\vert x_{N-1}\right\vert ^{\alpha \gamma })\left\vert
n_{j}(x)\right\vert dx\leq C\int\limits_{R^{N}}(1+x^{2})^{\frac{\gamma
+r}{2}}\left\vert n_{j}(x)\right\vert (1+x^{2})^{-\frac{r}{2}}dx\leq
\end{equation*}

\begin{equation*}
\leq C\left( \int\limits_{R^{N}}\left[ (1+x^{2})^{\frac{\gamma
+r}{2}}\left\vert n_{j}(x)\right\vert \right] ^{p^{\prime }}dx\right)
^{\frac{1}{p^{\prime }}}\left( \int\limits_{R^{N}}(1+x^{2})^{-\frac{rp}{2}}dx\right)
^{\frac{1}{p}}\leq
\end{equation*}

\begin{equation*}
\leq C\left[ \int\limits_{R^{N}}\left( \sum\limits_{\left\vert \omega
\right\vert =0}^{\gamma +r}\left\vert D_{\xi }^{\omega
}\widetilde{n}_{j}(\xi )\right\vert ^{p}\right) d\xi \right] ^{\frac{1}{p}}\leq
C\left\Vert \widetilde{m}(A_{j}\xi )\right\Vert _{W_{p}^{\gamma +r}(B_{0})}.
\end{equation*}

 The lemma follows.
\end{proof}

The same idea that was used in the proof of Lemma \ref{L2.1} can be
used by groups of the variables. That is, for example, we obtain
with $r>(N-1)/p$

\begin{equation*}
\int\limits_{R^{N}}(1+\left\vert x^{\prime \prime }\right\vert ^{\gamma
}+\left\vert x_{N-1}\right\vert ^{\alpha \gamma })\left\vert
n_{j}(x)\right\vert dx\leq C\int\limits_{R^{N}}(1+(x^{\prime
})^{2})^{\frac{\gamma }{2}}\left\vert n_{j}(x)\right\vert dx=
\end{equation*}

\begin{equation*}
=C\int\limits_{R^{N}}\left[ (1+(x^{\prime })^{2})^{\frac{\gamma
+r}{2}}(1+ix_{N})\left\vert n_{j}(x)\right\vert \right] \left[ (1+(x^{\prime
})^{2})^{-\frac{r}{2}}(1+ix_{N})^{-1}\right] dx\leq
\end{equation*}

\begin{equation*}
\leq C\left\{ \int\limits_{R^{N}}\left[ (1+(x^{\prime })^{2})^{\frac{\gamma
+r}{2}}(1+ix_{N})\left\vert n_{j}(x)\right\vert \right] ^{p^{\prime
}}dx\right\} ^{\frac{1}{p^{\prime }}}\left\{
\int\limits_{R^{N}}(1+(x^{\prime
})^{2})^{-\frac{rp}{2}}(1+ix_{N})^{-p}dx\right\} ^{\frac{1}{p}}\leq
\end{equation*}

\begin{equation*}
\leq C\sum\limits_{\substack{ |\omega ^{\prime }|\leq \gamma +r,  \\
\omega _{N}\in \left\{ 0,1\right\} }}\left\Vert D_{\xi ^{\prime
},\xi _{N}}^{(\omega ^{\prime },\omega _{N})}\widetilde{m}(A_{j}\xi
)\right\Vert _{L_{p}(B_{0})},
\end{equation*}
where the sum in the last expression is considered in all
multi-indices $\omega = (\omega^{\prime}, \omega_{N})$ such that
$|\omega^{\prime}| \leq \gamma + r $ with $ r> (N-1) / p $ and
$\omega_{N}\in \left\{0,1 \right\}$.

Thus the following lemma holds.

\begin{lemma}
\label{L2.2} Let for any  $\lambda > 0$ and for some $p\in (1,2]$
with $s>(N-1)/p$ we have

\begin{equation*}
M_{1}\equiv \sup\limits_{\lambda >0}\sum\limits_{\substack{ |\omega
^{\prime }|\leq \gamma +s,  \\ \omega _{N}\in \left\{ 0,1\right\}
}}\left\Vert D_{\xi ^{\prime },\xi _{N}}^{(\omega ^{\prime },\omega
_{N})}\widetilde{m}(A_{\lambda }\xi )\right\Vert
_{L_{p}(B_{0})}<\infty .
\end{equation*}

Then condition \eqref{A.02} is satisfied and $\mu \leq CM_{1}.$
\end{lemma}

Formulate for example yet another assertion, which can be proved
exactly the same way as the previous two lemmas taking into account
that $(1+\left\vert x^{\prime \prime }\right\vert ^{\gamma
}+\left\vert x_{N-1}\right\vert ^{\alpha \gamma })\leq
\prod_{k=1}^{N-1}|1+ix_{k}|$ and using the multiplication and
division by $\prod_{k=1}^{N}|1+ix_{k}|$.

\begin{lemma}
\label{L2.3} Suppose that for some $p\in (1,2]$ the following
condition is satisfied

\begin{equation*}
M_{2}\equiv \sup\limits_{\lambda >0}\sum\limits_{\omega }\left\Vert D_{\xi
}^{\omega }\widetilde{m}(A_{\lambda }\xi )\right\Vert _{L_{p}(B_{0})}<\infty
,
\end{equation*}
where the sum is taken over all multi-indices $\omega =(\omega
_{1},...,\omega _{N-1},\omega _{N})$ such that $\omega
_{1},...,\omega _{N-1}\in \left\{ 0,1,2\right\} $, $\omega _{N}\in
\left\{ 0,1\right\} $.

Then condition \eqref{A.02} is satisfied and $\mu \leq CM_{2}$.
\end{lemma}

The fact that considered in Theorem \ref{T2.1} multiplier
$\widetilde{m}(\xi)$ uses the smoothness of the density $u(x)$ for
all variable $x^{\prime}$ except for one variable $x_{N}$ is
insignificant and was considered only for simplicity. Directly from
the proof of Theorem \ref{T2.1} it follows that completely analogous
to this proof the following assertion can be proved.


Let a function $\widetilde{m}(\xi )\in C(R^{N}\backslash \{0\})$ be
bounded.  Let $x\in R^{N}$, let $K\in (0,N)$ be an integer,
$x=(x^{(1)},x^{(2)})$, $x^{(1)}=(x_{1},...,x_{K})$,
$x^{(2)}=(x_{K+1},...,x_{N})$ and similarly $\xi =(\xi ^{(1)},\xi
^{(2)})$, $\xi ^{(1)}=(\xi _{1},...,\xi _{K})$, $\xi ^{(2)}=(\xi
_{K+1},...,\xi _{N})$. Let $\alpha =(\alpha _{1},...,\alpha _{K})$,
$\beta =(\beta _{K+1},...,\beta _{N})$, $\alpha _{i},\in (0,1]$,
$\beta _{k}>0$, and  $\gamma \in (0,1)$.

Denote for $x\in R^{N}$ and for an integer $j\in Z$

\begin{equation}
A_{j}x\equiv (2^{\frac{j}{\alpha _{1}}}x_{1},...,2^{\frac{j}{\alpha
_{K}}}x_{K},2^{\frac{j}{\beta _{K+1}}}x_{K+1},2^{\frac{j}{\beta
_{N}}}x_{N}),\quad a_{j}=\det A_{j}.  \label{A.8.1.1}
\end{equation}

Denote as above $\widetilde{m}_{j}(\xi )=\widetilde{m}(\xi )\chi
(A_{j}^{-1}\xi )$, and let $m_{j}(x)$ be the inverse Fourier
transform of the function $\widetilde{m}_{j}(\xi )$,

\begin{equation}
n_{j}(x)=a_{j}^{-1}m_{j}(A_{j}^{-1}x).  \label{A.8.2.1}
\end{equation}

Let with some $\mu >0$ the following conditions are satisfied

\begin{equation}
\widetilde{m}(\xi )|_{\xi ^{(1)}=0}=\widetilde{m}(0,\xi ^{(2)})\equiv
0,\quad \xi ^{(2)}\in R^{N-K},  \label{A.01.1}
\end{equation}

\begin{equation}
\int\limits_{R^{N}}(1+\sum\limits_{k=1}^{K}|x_{k}|^{\alpha _{k}\gamma
})|n_{j}(x)|dx\leq \mu ,\quad j\in Z.  \label{A.02.1}
\end{equation}

Suppose finally that a function $u(x)\in C_{x^{(1)}}^{\alpha \gamma
}(R^{N})$ has a finite support in $R^{N}$ and satisfies
H\"{o}lder condition with respect to a part of the variables

\begin{equation*}
\left\langle u\right\rangle _{x^{(1)},R^{N}}^{(\alpha \gamma
)}=\sum\limits_{k=1}^{K}\left\langle u\right\rangle _{x_{k},R^{N}}^{(\alpha
_{k}\gamma )}<\infty .
\end{equation*}

Denote as above

\begin{equation*}
v(x)\equiv Mu\equiv m(x)\ast u(x)\equiv F^{-1}(\widetilde{m}(\xi
)\widetilde{u}(\xi )).
\end{equation*}

\begin{theorem}
\label{T2.2} Let conditions \eqref{A.01.1}, \eqref{A.02.1} are
satisfied. Then the function $v(x)$ satisfies the H\"{o}lder
condition with respect to all variables and

\begin{equation}
\left\langle v\right\rangle _{x^{(1)},x^{(2)},R^{N}}^{(\alpha \gamma ,\beta
\gamma )}\leq C\mu \left\langle u\right\rangle _{x^{(1)},R^{N}}^{(\alpha
\gamma )},  \label{A.03.1}
\end{equation}
where

\begin{equation*}
\left\langle v\right\rangle _{x^{(1)},x^{(2)},R^{N}}^{(\alpha \gamma ,\beta
\gamma )}=\sum\limits_{k=1}^{K}\left\langle v\right\rangle
_{x_{k},R^{N}}^{(\alpha _{k}\gamma )}+\sum\limits_{k=K+1}^{N}\left\langle
v\right\rangle _{x_{k},R^{N}}^{(\beta _{k}\gamma )}.
\end{equation*}
\end{theorem}

Completely analogous to the proof of Lemmas \ref{L2.1}, \ref{L2.2} a
sufficient condition for inequalities \eqref{A.02.1} can be
obtained. Similarly with Lemma \ref{L2.1} we have the following
assertion.

For $\lambda >0$ denote \ $A_{\lambda }\xi =(\lambda
^{\frac{1}{\alpha _{1}}}\xi _{1},...,\lambda ^{\frac{1}{\alpha
_{K}}}\xi _{K},\lambda ^{\frac{1}{\beta _{K+1}}}\xi
_{K+1},...,\lambda ^{\frac{1}{\beta _{N}}}\xi _{N})$ and denote
$B_{0}=\left\{ \xi \in R^{N}:1/8\leq \rho (\xi )\leq 8\right\} $, \
where $\rho (\xi )=\sum_{л =1}^{K}|\xi _{k}|^{\alpha _{k}}+\sum_{л
=K+1}^{N}|\xi _{k}|^{\beta _{k}}$.

\begin{lemma}
\label{L2.3.1} Let uniformly in  $\lambda >0$

\begin{equation*}
\widetilde{m}(A_{\lambda }\xi )\in W_{p}^{s}(B_{0}),\quad p\in (1,2],\quad
s>\frac{N}{p}+\gamma .
\end{equation*}

Then conditions \eqref{A.02.1} are satisfied and

\begin{equation*}
\mu \leq \sup\limits_{\lambda }C\left\Vert \widetilde{m}(A_{\lambda }\xi
)\right\Vert _{W_{p}^{s}(B_{0})}.
\end{equation*}
\end{lemma}

We also have a more general assertion similar to Theorem \ref{T1.1}.

Denote similarly to the previous section the spaces

\begin{equation*}
\mathcal{H}_{x^{(1)}}^{\alpha \gamma }(R^{N})=C_{x^{(1)}}^{\alpha
\gamma }(R^{N})\cap L_{2}(R^{N}),\quad
\mathcal{H}_{x^{(1)},x^{(2)}}^{\alpha \gamma ,\beta \gamma
}(R^{N})=C_{x^{(1)},x^{(2)}}^{\alpha \gamma, \beta \gamma
}(R^{N})\cap L_{2}(R^{N}),
\end{equation*}
which are the closures of the set of finite functions in the norms

\begin{equation*}
\left\Vert u\right\Vert _{\mathcal{H}_{x^{(1)}}^{\alpha \gamma
}(R^{N})}\equiv \left\Vert u\right\Vert _{L_{2}(R^{N})}+\left\langle
u\right\rangle _{x^{(1)},R^{N}}^{(\alpha \gamma )},\quad \left\Vert
u\right\Vert _{\mathcal{H}_{x^{(1)},x^{(2)}}^{\alpha \gamma ,\beta \gamma
}(R^{N})}\equiv \left\Vert u\right\Vert _{L_{2}(R^{N})}+\left\langle
u\right\rangle _{x^{(1)},x^{(2)},R^{N}}^{(\alpha \gamma ,\beta \gamma )}.
\end{equation*}

Closure of estimate \eqref {A.03.1}, proved for a finite function
$u(x)$, in the norms of the spaces  $\mathcal{H}_{x^{(1)}}^{\alpha
\gamma }(R^{N})$,  $\mathcal{H}_{x^{(1)},x^{(2)}}^{\alpha \gamma
,\beta \gamma }(R^{N})$ and the using of the scheme of the proofs of
Lemmas \ref{L2.1}, \ref{L2.2} leads to the following assertin.

\begin{theorem}
\label{T2.3} Let condition \eqref{A.01.1} be satisfied. Let further
in the notation of Theorem \ref{T1.1} instead of condition
\eqref{1.7} the following condition be satisfied
\begin{equation}
s_{i}>\frac{N_{i}}{p}+\gamma ,\quad i=\overline{1,r},\quad p\in (1,2].
\label{1.7.1}
\end{equation}
Let also similar to \eqref{1.8} the following condition be satisfied
\begin{equation}
\sup_{\lambda >0}\sum_{\left\vert \omega _{i}\right\vert \leq
s_{i}}\left\Vert D_{y_{1}}^{\omega _{1}}D_{y_{2}}^{\omega
_{2}}...D_{y_{r}}^{\omega _{r}}\widetilde{m}(A_{\lambda }\xi )\right\Vert
_{L_{p}(B_{\nu })}\leq \mu .  \label{1.8.1}
\end{equation}

Then the operator $M$ is a bounded linear operator from
$\mathcal{H}_{x^{(1)}}^{\alpha \gamma }(R^{N})$ to
$\mathcal{H}_{x^{(1)},x^{(2)}}^{\alpha \gamma ,\beta \gamma
}(R^{N})$ and
\begin{equation}
\left\Vert Mu\right\Vert _{H_{x^{(1)},x^{(2)}}^{\alpha \gamma ,\beta \gamma
}(R^{N})}\leq C\mu \left\Vert u\right\Vert _{H_{x^{(1)}}^{\alpha \gamma
}(R^{N})}.  \label{A.03.2}
\end{equation}
\end{theorem}



In the following section, we will demonstrate the using of the
proven statements about multipliers to an initial-boundary value
problems for the linearized Cahn-Hilliard equation with dynamic
boundary conditions of two types. Here we give simple examples of
applications of Theorems \ref{T2.1}, \ref{T2.3}.

\bigskip

\textbf{Example 1.}

\bigskip

Let a function $u(x)$ has compact support in $R^{N}$ and satisfies
the Poisson equation

\begin{equation}
\Delta u(x)=f(x),  \label{Ex.0}
\end{equation}

where a function $f(x) $ has compact support in $R^{N}$ and
satisfies H\"{o}lder condition for some single variable, for
example, $x_{1}$ with an exponent $\gamma \in (0,1)$

\begin{equation}
\left\langle f\right\rangle _{x_{1}}^{(\gamma
)}=\sup_{h>0}\frac{|f(x+h\overrightarrow{e}_{1})-f(x)|}{h^{\gamma }}<\infty .  \label{Ex.01}
\end{equation}

Consider all the second derivatives of $u(x)$ containing derivative
with respect to $x_{1}$. It is well known that in terms of Fourier
transform we have the equality

\begin{equation*}
\widetilde{\frac{\partial ^{2}u}{\partial x_{k}\partial x_{1}}}(\xi
)=C\frac{\xi _{k}\xi _{1}}{\xi ^{2}}\widetilde{f}(\xi ),\quad k=\overline{1,N}.
\end{equation*}

Since the function $\widetilde{m}(\xi)=$ $\frac{\xi _{k}
\xi_{1}}{\xi^{2}} $ has the property $\widetilde{m}(\lambda \xi) =
\widetilde{m}(\xi)$ for any positive $\lambda$ and is smooth away
from the origin, then it is  easy to verify the conditions of
Theorem \ref{T2.3}. Consequently,

\begin{equation}
\left\langle \frac{\partial ^{2}u}{\partial x_{k}\partial
x_{1}}\right\rangle _{x}^{(\gamma )}\leq C\left\langle f\right\rangle
_{x_{1}}^{(\gamma )},\quad k=\overline{1,N},  \label{Ex.1}
\end{equation}

where the H\"{o}lder constant in the left hand side of this
inequality is taken with respect to all variables, not only with
respect to $x_{1}$.

Note that in estimate \eqref{Ex.1} only the second derivatives
containing the derivative with respect to $x_{1}$ are present. This
fact is essential as the following example shows. Let $\eta (x)\in
C_{0}^{\infty }(R^{3})$. Consider the function

\begin{equation*}
u_{1}(x)=u_{1}(x_{1},x_{2},x_{3})=(x_{2}^{2}-x_{3}^{2}+x_{2}x_{3})\ln
(x_{2}^{2}+x_{3}^{2})\eta (x_{1},x_{2},x_{3}).
\end{equation*}

It is immediately verified that the function with compact support
$u_{1}(x)$ satisfies equation \eqref{Ex.0} with  right-hand side
satisfying \eqref{Ex.01}. However, it's second derivatives that do
not contain the derivative with respect to$x_{1}$ not only do not
satisfy H\"{o}lder condition, but are just unbounded in
neighborhood of zero.

This example shows that condition \eqref{A.01} on a multiplier can
not be generally dropped. Although the author does not known how
close it is to the sharp condition.

\bigskip

\textbf{Example 2.}

\bigskip

It is interesting, in our opinion, to  consider the following simple
example for a parabolic equation. Let a function $u(x,t)$ with
compact support in $R^{N}\times R^{1}$ satisfies the heat equation

\begin{equation}
\frac{\partial u}{\partial t}-\Delta u=f(x,t),  \label{Ex.2}
\end{equation}

where the fight hand side $f(x,t)$ with compact support satisfies
the H\"{o}lder condition with respect to the variable $t$ only
with the exponent greater than $1/2 $, that is

\begin{equation*}
\left\langle f(x,t)\right\rangle _{t,R^{N}\times R^{1}}^{(\gamma )}<\infty
,\quad \gamma \in (\frac{1}{2},1).
\end{equation*}

Making in \eqref{Ex.2} Fourier transform and denoting the variable
of the Fourier transform with respect to $t$ by $\xi_{0}$, we obtain

\begin{equation*}
\widetilde{\frac{\partial u}{\partial t}}=C \frac{i\xi _{0}}{i\xi _{0}+\xi
^{2}}\widetilde{f}(\xi ,\xi _{0}).
\end{equation*}

Then it follows from Theorem \ref{T2.3} that

\begin{equation*}
\left\langle \frac{\partial u}{\partial t}\right\rangle
_{t,R^{N}\times R^{1}}^{(\gamma )}+\left\langle \frac{\partial
u}{\partial t}\right\rangle _{x,R^{N}\times R^{1}}^{(2\gamma )}\leq
C\left\langle f(x,t)\right\rangle _{t,R^{N}\times R^{1}}^{(\gamma
)}.
\end{equation*}

In particular, since $2\gamma \in (1,2)$, the derivative
$\frac{\partial u}{\partial t}$ has derivatives with respect to
$x_{i}$ and

\begin{equation*}
\sum\limits_{i=1}^{N}\left\langle \frac{\partial ^{2}u}{\partial t\partial
x_{i}}\right\rangle _{x,R^{N}\times R^{1}}^{(2\gamma -1)}\leq C\left\langle
f(x,t)\right\rangle _{t,R^{N}\times R^{1}}^{(\gamma )}.
\end{equation*}

Note again that, as in the previous example, the second derivatives
with respect to the variables $x_{i}$ can be unbounded in general,
for example, $u(x_{1},x_{2},t)=(x_{1}^{2}-x_{2}^{2}+x_{2}x_{1})\ln
(x_{1}^{2}+x_{2}^{2})\eta (x_{1},x_{2})\psi (t)$, $\eta \in
C_{0}^{\infty }(R^{2})$, $\psi \in C_{0}^{\infty }(R^{1})$.

\section{Model problems in a half-space for the linearized
Cahn-Hilliard equation with dynamic boundary conditions.}

\label{s3}

In this section we consider the Schauder estimates for initial
boundary value problems in a half-space for the linearized
Cahn-Hilliard equation with dynamic boundary conditions. These
problems are not included in the well-known general theory of
parabolic initial-boundary value problems (see, eg, \cite{6} -
\cite{6.2}). However, we  significantly use the results of \cite{6}.

The presentation in this section is very sketchy. More detailed text
will be given in a forthcoming paper.

Define the used below space of smooth functions. Let $\Omega $ be a
domain in $R^{N}$, which can be unbounded. Denote by $\Omega
_{T}=\Omega \times (0,T)$, where $T>0$ or $T=+\infty $ . We use
Banach functional spaces $C^{l_{1},l_{2}}(\overline{\Omega _{T}})$
 with elements $u(x,t)$, $x\in \overline{\Omega }$, $t\in \lbrack 0,T]$,
$l_{1},l_{2}>0$ are non-integer. These spaces are defined, for
example, in \cite{Sol15} and consist of functions with smoothness
with respect to the variables $x$ up to the order $l_{1}$ and with
smoothness with respect to the variable $t$ up to the order $l_{2}$,
ie, having a finite norm

\begin{equation}
|u|_{\overline{\Omega _{T}}}^{(^{l_{1},l_{2}})}\equiv \left\Vert
u\right\Vert _{C^{l_{1},l_{2}}(\overline{\Omega _{T}})}\equiv
|u|_{\overline{\Omega _{T}}}^{(0)}+\sum\limits_{|\alpha |=[l_{1}]}\left\langle
D_{x}^{\alpha }u\right\rangle _{x,\overline{\Omega
_{T}}}^{(l_{1}-[l_{1}])}+\left\langle D_{t}^{[l_{2}]}u\right\rangle _{t,\overline{\Omega
_{T}}}^{(l_{2}-[l_{2}])}.  \label{3.1}
\end{equation}

Here $\alpha =(\alpha _{1},...,\alpha _{N})$ is a multiindex,
$|\alpha |=\alpha _{1}+...+\alpha _{N}$, $D_{x}^{\alpha
}=D_{x_{1}}^{\alpha _{1}}...D_{x_{N}}^{\alpha _{N}}$, $[l]$ is the
integer part of a number $l$, $|u|_{\overline{\Omega
_{T}}}^{(0)}=\max_{\overline{\Omega _{T}}}|u(x,t)|$, $\left\langle
D_{x}^{\alpha }u\right\rangle _{x,\overline{\Omega
_{T}}}^{(l_{1}-[l_{1}])}$, $\left\langle
D_{t}^{[l_{2}]}u\right\rangle _{t,\overline{\Omega
_{T}}}^{(l_{2}-[l_{2}])}$ are H\"{o}lder constants of the
corresponding functions with respect to $x$ and $t$ correspondingly
over a domain$\overline{\Omega _{T}}$. Besides quantities in
\eqref{3.1}, for functions from the space
$C^{l_{1},l_{2}}(\overline{\Omega _{T}})$ the H\"{o}lder
constants of the derivatives $D_{x}^{\alpha }u$ with respect to $t$
are finite with some exponents and the same is true  for the
H\"{o}lder constants of the derivatives $D_{t}^{[l_{2}]}u$ with
respect to $x$ and for mixed derivatives up to some order. Estimates
of all these H\"{o}lder constants are obtained by interpolation
with the using of \eqref{3.1} - see, for example \cite{Sol15}. Below
we will use the space $C^{l,l/4}(\overline{\Omega_{T}})$, where $l$
is a non-integer positive number and the norm in this space we will
denote for simplicity by $|u|^{(l)}_{\overline{\Omega}_{T}}$.

We will use also the spaces $C_{0}^{l_{1},l_{2}}(\overline{\Omega
_{T}})$, where zero at the bottom of the notation denotes a closed
subspace of $C^{l_{1}, l_{2}}(\overline{\Omega_{T}})$, consisting of
functions whose derivatives with respect to $t$ up to the order
$[l_{2}]$ vanish identically at $t =0$. The functions of these
spaces can be considered to be extended identically zero to $t\leq
0$ with the preservation of the class.

We proceed to the formulation of the problem. Denote
$Q_{+}^{N+1}=\left\{ (x,t)\in R^{N}\times R^{1}:x_{N}\geq 0,t\geq
0\right\} $, $Q_{+,T}^{N+1}=\left\{ (x,t)\in R^{N}\times
R^{1}:x_{N}\geq 0,t\in \lbrack 0,T]\right\} $ , $Q_{+}^{N}=\left\{
(x,t)\in R^{N}\times R^{1}:x_{N}=0,t\geq 0\right\} $, \
$Q_{+,T}^{N}=\left\{ (x,t)\in R^{N}\times R^{1}:x_{N}=0,t\in \lbrack
0,T]\right\} $\ , $Q^{N}=$\ $\left\{ (x,t)\in R^{N}\times
R^{1}:x_{N}=0\right\} =R^{N-1}\times R^{1}$\ ,$x=(x^{\prime
},x_{N})$. Consider in $Q_{+}^{N +1}$ the following initial boundary
value problem for the unknown function $u (x,t)$:

\begin{equation}
\frac{\partial u}{\partial t}+\Delta ^{2}u=f(x,t),\quad (x,t)\in Q_{+}^{N+1},
\label{M.2}
\end{equation}

\begin{equation}
\frac{\partial \Delta u}{\partial x_{N}}=g(x^{\prime },t),\quad (x^{\prime
},t)\in Q_{+}^{N},  \label{M.3}
\end{equation}

\begin{equation}
u(x,0)=0,\quad x_{N}\geq 0,  \label{M.1}
\end{equation}

\begin{equation}
\frac{\partial u}{\partial t}-a\Delta _{x^{\prime }}u=h_{1}(x^{\prime
},t),\quad (x^{\prime },t)\in Q_{+}^{N},  \label{M.4.1}
\end{equation}
where $\Delta $ is the Laplace operator, $\Delta _{x^{\prime }}$ is
the Laplace operator with respect to the variables $x^{\prime }$,
$a$ is a positive constant, and we assume that the function $u(x,t)$
is bounded at $|x|\rightarrow \infty $. Together with boundary
dynamic condition \eqref{M.4.1} (instead of this condition
conditions) we also consider other boundary condition

\begin{equation}
\frac{\partial u}{\partial t}-a\frac{\partial u}{\partial
x_{N}}=h_{2}(x^{\prime },t),\quad (x^{\prime },t)\in Q_{+}^{N}.  \label{M.4.2}
\end{equation}

The physical meaning of the condition of the form \eqref{M.4.1} is
explained, for example, in \cite{8}, and the condition \eqref{M.4.2}
is explained, for example, in \cite{7}. In this case, in \cite{7}
was considered a more general boundary condition

\begin{equation*}
\frac{\partial u}{\partial t}-a\Delta _{x^{\prime
}}u-b\frac{\partial u}{\partial x_{N}}=h(x^{\prime },t),\quad
(x^{\prime },t)\in Q_{+}^{N}.
\end{equation*}
But (at least when considering the classes of smooth functions) the
term $b \frac{\partial u}{\partial x_{N}}$ is in this condition a
junior (in order) term  and can be omitted when considering the
model problem.

We assume that the given functions $f$, $g$, $h_{1}$, $h_{2}$ have
compact supports and belong  to the following spaces with zero at
the bottom with some $\gamma \in (0,1)$

\begin{equation}
f\in C_{0}^{\gamma ,\frac{\gamma }{4}}(Q_{+}^{N+1}),\quad g\in
C_{0}^{1+\gamma ,\frac{1+\gamma }{4}}(Q_{+}^{N}),\quad h_{1}\in
C_{0}^{2+\gamma ,\frac{2+\gamma }{4}}(Q_{+}^{N}),\quad h_{2}\in
C_{0}^{3+\gamma ,\frac{3+\gamma }{4}}(Q_{+}^{N}).  \label{M.5}
\end{equation}
The solution $u(x,t)$ we will suppose in the class $C_{0}^{4 +
\gamma, \frac{4 + \gamma}{4}}(Q_{+}^{N +1})$, that is dictated by
the anisotropy of equation \eqref{M.2}. But besides we require that
$u_{t}(x^{\prime}, 0, t) \in C_{0}^{2 + \gamma, \frac{2 +
\gamma}{4}}(Q_{+}^{N})$ or $u_{t}(x^{\prime },0,t)\in
C_{0}^{3+\gamma ,\frac{3+\gamma }{4}}(Q_{+}^{N})$ depending on the
type of dynamic boundary conditions.

Note that in view of \eqref{M.1}, \eqref{M.2}  and the fact that the
given functions $ f $, $ g $, $h_{1}$, $h_{2}$ belong to the spaces
with zero at the bottom the function $u(x,t)$ must e satisfy the
condition $\partial u / \partial t (x,0) \equiv 0$. Together with
\eqref{M.1} this allows to consider the function $u(x,t)$ and all
the given functions $f$, $g$, $h_{1}$, $h_{2}$ to be extended by
zero to $t <0$ and consider relations \eqref{M.2} - \eqref{M.4.2}
for all values of the time variable $t \in R^{1}$.

\subsection{Problem \eqref{M.2} - \eqref{M.4.1}.}

\label{ss3.1}

Consider problem \eqref{M.2} - \eqref{M.4.1}. Denote

\begin{equation}
\rho (x^{\prime },t)\equiv u(x^{\prime },0,t)=u(x,t)|_{x_{N}=0}.
\label{M.11}
\end{equation}
Condition \eqref{M.4.1}  allows to find the value of the unknown
function $u(x,t)$ at $x_{N}=0$, that is the function $\rho
(x^{\prime },t)$, namely,

\begin{equation}
\rho (x^{\prime },t)=\Gamma _{a}(x^{\prime },t)\ast h_{1}(x^{\prime },t),
\label{M.17}
\end{equation}

where $\Gamma _{a}(x^{\prime },t)$ is the fundamental solution of
the heat operator $L_{a}\equiv
\partial /\partial t-a\Delta _{x^{\prime }}$. It is well known that
expression \eqref{M.17} can be obtained from \eqref{M.4.1} by
applying the Fourier transform with respect to the variables
$x^{\prime}$ and $t$. In other words,  denoting for a function
$v(x^{\prime },t)$

\begin{equation*}
\widetilde{v}(\xi _{0},\xi )=\int\limits_{-\infty }^{+\infty
}dt\int\limits_{R^{N-1}}e^{-i\xi _{0}t-i\xi x^{\prime }}v(x^{\prime
},t)dx^{\prime }
\end{equation*}
and applying this transform to relation \eqref{M.4.1} (recall that
all the functions are assumed to be extended by zero to $t<0$), in
view of the known properties of the Fourier transform of
derivatives, we find that

\begin{equation}
\widetilde{\rho }(\xi _{0},\xi )=\frac{\widetilde{h_{1}}(\xi _{0},\xi
)}{i\xi _{0}+a\xi ^{2}}.  \label{M.16}
\end{equation}

Estimates for the potential in \eqref{M.17} are well known in the
case when its density $h_{1}\in C^{k+\gamma ,\frac{k+\gamma
}{2}}(Q_{+}^{N})$.  However, in our case we are dealing with another
anisotropy of smoothness of the space for the density, namely
$h_{1}\in C_{0}^{2+\gamma ,\frac{2+\gamma }{4}}(Q_{+}^{N})$.
Therefore  known properties of the potential for the heat operator
are inapplicable in our case. Furthermore, the results of \cite{1}
and Theorem \ref{T1.1} also are inapplicable, since the anisotropy
of homogeneity of the kernel does not coincide with the anisotropy
of the smoothness of the density $h_{1}$. Therefore, we obtain
estimates of the H\"{o}lder  constants for highest derivatives
of the function $\rho(x^{\prime}, t)$ in the space $C^{4 + \gamma,
\frac{4 + \gamma}{4}}(Q_{+}^{N})$  using Theorem \ref{T2.3}.

Consider first the H\"{o}lder constant in the variable $t$ of
the derivative $\rho_{t}(x^{\prime}, t)$. In view of relation
\eqref{M.16} and known properties of the Fourier transform of
derivatives

\begin{equation}
\widetilde{\rho _{t}}=\frac{i\xi _{0}}{i\xi _{0}+a\xi
^{2}}\widetilde{h_{1}}(\xi _{0},\xi ).  \label{M.18}
\end{equation}

Consider the function

\begin{equation}
\widetilde{m}_{1}(\xi _{0},\xi )=\frac{i\xi _{0}}{i\xi _{0}+a\xi
^{2}}. \label{M.19}
\end{equation}

Evidently this function is homogeneous of degree zero
\begin{equation}
\widetilde{m}_{1}(\lambda ^{2}\xi _{0},\lambda \xi
)=\widetilde{m}_{1}(\xi _{0},\xi ),\quad \lambda >0.  \label{M.20}
\end{equation}

Besides this function is smooth on the set $B_{1}=\left\{ (\xi
_{0},\xi ):1/8<|\xi _{0}|+\xi ^{2}<8\right\}$. Therefore it is
trivial to verify that $\widetilde{m}_{1}(\xi _{0},\xi )$ satisfies
the condition of theorem \ref{T2.3}. Consequently
\begin{equation}
\left\langle \rho _{t}\right\rangle _{t,Q_{+}^{N}}^{(\frac{2+\gamma
}{4})}\leq C\left\langle h_{1}\right\rangle _{t,Q_{+}^{N}}^{(\frac{2+\gamma
}{4})}.  \label{M.21}
\end{equation}

Similar estimates of other derivatives of $\rho$ result in the
estimate

\begin{equation}
\left\vert \rho \right\vert _{C^{4+\gamma ,\frac{4+\gamma
}{4}}(Q_{+,T}^{N})}+\left\vert \rho _{t}\right\vert _{C^{2+\gamma
,\frac{2+\gamma }{4}}(Q_{+,T}^{N})}\leq C_{T}\left\vert h_{1}\right\vert
_{C^{2+\gamma ,\frac{2+\gamma }{4}}(Q_{+,T}^{N})}.  \label{M.25}
\end{equation}

Thus in problem \eqref{M.1}-\eqref{M.4.1} condition \eqref{M.4.1}
can be replaced by the condition

\begin{equation}
u(x^{\prime },0,t)=\rho (x^{\prime },t),  \label{M.26}
\end{equation}
where for the function $\rho (x^{\prime },t)$ estimate \eqref{M.25}
is valid. Then from the results of \cite{6} it follows that this
problem has a unique solution $u(x,t)$ and

\begin{equation}
\left\vert u\right\vert _{C^{4+\gamma ,\frac{4+\gamma
}{4}}(Q_{+,T}^{N+1})}\leq C_{T}\left( \left\vert f\right\vert _{C^{\gamma
,\frac{\gamma }{4}}(Q_{+,T}^{N+1})}+\left\vert g\right\vert _{C^{1+\gamma
,\frac{1+\gamma }{4}}(Q_{+,T}^{N})}+\left\vert h_{1}\right\vert _{C^{2+\gamma
,\frac{2+\gamma }{4}}(Q_{+,T}^{N})}\right) ,  \label{M.27}
\end{equation}
and besides in view of \eqref{M.25}

\begin{equation}
\left\vert u_{t}(x^{\prime },0,t)\right\vert _{C^{2+\gamma ,\frac{2+\gamma
}{4}}(Q_{+,T}^{N})}\leq C_{T}\left\vert h_{1}\right\vert _{C^{2+\gamma
,\frac{2+\gamma }{4}}(Q_{+,T}^{N})}.  \label{M.28}
\end{equation}

Thus we have proved the following assertion.

\begin{theorem}
\label{T3.1} Under conditions \eqref{M.5} and for any $T>0$ problem
\eqref{M.1}-\eqref{M.4.1} has the unique solution $u(x,t)$ from the
space $C^{4+\gamma ,\frac{4+\gamma }{4}}(Q_{+,T}^{N+1})$ and
estimates \eqref{M.27}, \eqref{M.28} are valid.
\end{theorem}

\subsection{Problem \eqref{M.2} - \eqref{M.1}, \eqref{M.4.2}.}

\label{ss3.2}

As in the previous section, we reduce the problem to a problem with
condition \eqref{M.26} instead of condition \eqref{M.4.2} after determining
the function $\rho (x^{\prime },t)\equiv u(x^{\prime },0,t)$ from the conditions of
the problem. However, in this case the boundary operator in the left side of \eqref{M.4.2} is not
a local operator (as opposed to \eqref{M.4.1}), so its consideration
requires somewhat more complex reasoning. This is due to the fact that
in this case a more complex multiplier arises, which is not a
homogeneous function. To study this multiplier, we extract it's "the main"
homogeneous part.

Using well-known results on the solvability of parabolic boundary value problems, we can
reduce problem \eqref{M.2} - \eqref{M.1}, \eqref {M.4.2} to the
case when $f \equiv 0$ and $g \equiv 0 $. Besides, we will denote for simplicity the function
$h_{2}$ as just $h$.

Make in problem \eqref{M.2} - \eqref{M.1}, \eqref{M.4.2}  the
Fourier transform with respect to the variables $x^{\prime}$ and $t$.
As a result these relations take the form

%
%
%



\begin{equation}
i\xi _{0}\widetilde{u}+\left( -\xi
^{2}+\frac{d^{2}}{dx_{N}^{2}}\right) ^{2}\widetilde{u}=0,\quad x_{N}>0,  \label{M2.19}
\end{equation}

\begin{equation}
\left. \frac{d}{dx_{N}}\left( -\xi ^{2}+\frac{d^{2}}{dx_{N}^{2}}\right)
\widetilde{u}\right\vert _{x_{N}=0}=0,  \label{M2.20}
\end{equation}

\begin{equation}
i\xi _{0}\widetilde{\rho }-a\left.
\frac{d\widetilde{u}}{dx_{N}}\right\vert _{x_{N}=0}=\widetilde{h},  \label{M2.21}
\end{equation}

\begin{equation}
\left\vert \widetilde{u}\right\vert \leq C,\quad x_{N}\rightarrow
\infty ,  \label{M2.22}
\end{equation}
where $\rho (x^{\prime },t)\equiv u(x^{\prime
},0,t)$.

From these relations we find

\begin{equation}
\widetilde{\rho }=\frac{\widetilde{h}(\xi ,\xi _{0})}{i\xi
_{0}+\frac{2\sqrt{\xi ^{2}+\sqrt[2]{-i\xi _{0}}}\sqrt{\xi ^{2}-\sqrt[2]{-i\xi
_{0}}}}{\sqrt{\xi ^{2}+\sqrt[2]{-i\xi _{0}}}+\sqrt{\xi ^{2}-\sqrt[2]{-i\xi
_{0}}}}}=\frac{\widetilde{h^{(k)}}(\xi ,\xi _{0})}{\widetilde{M}(\xi ,\xi _{0})}.
\label{M.43}
\end{equation}

Let $\alpha=(\alpha_{1},...,\alpha_{N-1})$, $\beta=(\beta_{1},...,\beta_{N-1})$ be multi-indexes,
 $|\alpha|=4$, $|\beta|=3$. By simple algebraic manipulations we can get from the last equality

\begin{equation}
\widetilde{D_{x^{\prime }}^{\alpha }\rho }=\frac{i\xi _{l}}{i\xi _{0}+|\xi
|}\widetilde{D_{x^{\prime }}^{\beta }h}+\mathbb{A}\rho,
\label{M.59}
\end{equation}

where $\mathbb{A}$ is a smoothing operator.

Using assertion about multipliers from the previous section we can eventually get the
estimate

\begin{equation}
\left\vert u\right\vert _{Q_{+,T}^{N+1}}^{(4+\gamma ,\frac{4+\gamma
}{4})}+\left\vert D_{t}u(x^{\prime },0,t)\right\vert _{Q_{+,T}^{N}}^{(3+\gamma
,\frac{3+\gamma }{4})}\leq C_{T}\left( \left\vert f\right\vert
_{Q_{+,T}^{N+1}}^{(\gamma ,\frac{\gamma }{4})}+\left\vert g\right\vert
_{Q_{+,T}^{N}}^{(3+\gamma ,\frac{3+\gamma }{4})}+\left\vert h\right\vert
_{Q_{+,T}^{N}}^{(3+\gamma ,\frac{3+\gamma }{4})}\right) .  \label{M.75}
\end{equation}


Thus we have the following theorem.



\begin{theorem}
\label{T3.2} Let $T>0$ be arbitrary and let for problem
\eqref{M.2} - \eqref{M.1}, \eqref{M.4.2} conditions \eqref{M.5} are satisfied.
Then this problem has the unique solution  $u(x,t)$ from the space
 $u(x,t)\in C^{4+\gamma ,\frac{4+\gamma }{4}}(Q_{+,T}^{N+1})$,
$u_{t}(x^{\prime },0,t)\in C^{3+\gamma ,\frac{3+\gamma }{4}}(Q_{+,T}^{N})$
and estimate  \eqref{M.75} is valid.
\end{theorem}

\end{document}